\pdfoutput=1
\documentclass[11pt,reqno]{amsart}
\usepackage[T1]{fontenc}
\usepackage{lmodern,amssymb,mathtools,booktabs,microtype}
\usepackage[a4paper,margin=28mm,headsep=6mm]{geometry}
\usepackage{tikz}
\usetikzlibrary{arrows.meta,calc,decorations.pathreplacing}
\usepackage[hidelinks]{hyperref}
\hypersetup{pdftitle={Arc-disjoint Hamiltonian paths in Cayley digraphs with a normal difference subgroup},pdfauthor={SangHyun Park},pdfsubject={Exact extension of coset assignments and confinement of Hamiltonian cut spectra}}
\newtheorem{theorem}{Theorem}[section]
\newtheorem{proposition}[theorem]{Proposition}
\newtheorem{lemma}[theorem]{Lemma}
\newtheorem{corollary}[theorem]{Corollary}
\theoremstyle{definition}
\newtheorem{definition}[theorem]{Definition}
\theoremstyle{remark}
\newtheorem{remark}[theorem]{Remark}

\numberwithin{equation}{section}
\newcommand{\Z}{\mathbb Z}
\newcommand{\Zq}{\mathbb Z_q}
\newcommand{\Cay}{\operatorname{Cay}}
\newcommand{\cyc}{\operatorname{cyc}}
\newcommand{\dist}{\operatorname{dist}}
\newcommand{\ord}{\operatorname{ord}}
\newcommand{\len}{\mathcal L}
\newcommand{\F}{\mathcal F}
\allowdisplaybreaks[1]
\title[Hamiltonian paths and prescribed coset choices]{Arc-disjoint Hamiltonian paths in Cayley digraphs with a normal difference subgroup}
\author{SangHyun Park}
\address{Yonsei University, Republic of Korea}
\date{}
\begin{document}
\begin{abstract}
Let $G=\langle a,b\rangle$ be a finite group with distinct nonidentity generators, and suppose that $H=\langle ba^{-1}\rangle$ is normal. We determine exactly which complementary generator choices on the nonterminal cosets of $H$ extend to two arc-disjoint Hamiltonian paths with both terminal vertices in $H$. The associated return residue must avoid an explicit set of at most six residues. When $[G:H]\ge3$, any two nonterminal cosets may be left free while all other choices are prescribed; two is best possible. Every such Cayley digraph has two arc-disjoint Hamiltonian paths; this includes all two-generator Cayley digraphs of finite abelian groups. The criterion follows from a complete distance formula for the Hamiltonian cut spectra of $\Cay(\Z_q;\{D,D+1\})$. Their unique nondegenerate exception comes from the odd central coefficient $(q-1)/2$. The geometric ingredient is a confinement theorem: two sector-filled integer sequences satisfying a length inequality, with overlapping ranges and distance at least two, have one sequence contained in a single gap of the other. Sector filling bounds alternation, and the length inequality excludes interlacing. A self-contained proof of the primitive-ray cycle profile follows the intersection geometry of Curran and Witte.
\end{abstract}
\maketitle

\section{Introduction}\label{sec:intro}
A Cayley digraph on $n$ vertices with two distinct nonidentity generators has $2n$ arcs. Two arc-disjoint Hamiltonian paths together cover all but two of these arcs. The distinction between path packing and cycle decomposition is essential: the directed torus $\vec C_3\square\vec C_4$ has two arc-disjoint Hamiltonian paths~\cite{DMM}, whereas it has no Hamiltonian cycle~\cite{TE}.

We study the choices that can be prescribed before the two paths are completed. For a finite group $G=\langle a,b\rangle$, we use the right Cayley convention, with arcs $x\to xa$ and $x\to xb$, and write
\[
 \eta=ba^{-1},\qquad H=\langle\eta\rangle.
\]
We call $H$ the difference subgroup of the specified generating pair. When $H$ is normal, the two paths can be constructed by choosing complementary generators on its nonterminal cosets and joining the resulting segments inside one terminal coset. Our results give both existence and an exact criterion for the prescribed choices.

\begin{theorem}\label{thm:existence}
Let $G=\langle a,b\rangle$ be a finite group, where $a,b$ are distinct nonidentity elements. If $\langle ba^{-1}\rangle$ is normal in $G$, then $\Cay(G;\{a,b\})$ has two arc-disjoint directed Hamiltonian paths.
\end{theorem}

The criterion below strengthens this statement. Put $q=|H|\ge2$ and $g=[G:H]$. Since $G/H$ is cyclic, there are residues $r,t\in\Zq=\Z/q\Z$ such that
\begin{equation}\label{eq:intro-coordinates}
 a\eta a^{-1}=\eta^r,\qquad a^g=\eta^t,\qquad
 r^g=1,\quad (r-1)t=0\pmod q.
\end{equation}
Every element has the form $\eta^z a^j$, with $0\le z<q$ and $0\le j<g$. We designate $H$ as the terminal coset. For $E\subseteq\{1,\ldots,g-1\}$, the first color uses $b$ throughout $Ha^j$ when $j\in E$ and $a$ otherwise; the second color makes the complementary choices. Define
\begin{equation}\label{eq:intro-state}
 B=2t+\sum_{j=0}^{g-1}r^j,\qquad
 u_E=t+\sum_{j\in E}r^j\pmod q.
\end{equation}
For a residue $B$, let $\F(q,B)$ be the union of the following indicated sets:
\begin{align}
 &\{0,B-1\} &&\text{if }\gcd(q,B-1)>2,\label{eq:Fminus}\\
 &\{-1,B\} &&\text{if }\gcd(q,B+1)>2.\label{eq:Fplus}
\end{align}
If $q$ is odd, also include
\begin{equation}\label{eq:Fcentral}
 \{c,B+c\},\qquad c=(q-1)/2,
 \quad\text{if }\gcd(q,2B-1)>1\text{ and }\gcd(q,2B+1)>1.
\end{equation}
All entries are residues modulo $q$. In particular, $\F(2,B)$ is empty.

\begin{theorem}[Exact extension of coset choices]\label{thm:extension}
Assume $g\ge2$. The complementary choices specified by $E$ extend to two arc-disjoint Hamiltonian paths with both terminal vertices in $H$ if and only if
\[
 u_E\notin\F(q,B).
\]
There are at most four exceptional residues for even $q$ and at most six for odd $q$.

If $g\ge3$, choose any two distinct indices $i,j\in\{1,\ldots,g-1\}$ and prescribe the choices on all other nonterminal cosets. If $E_0$ denotes those prescribed to use $b$ in the first color, then one of
\[
 E_0,\qquad E_0\cup\{i\},\qquad E_0\cup\{i,j\}
\]
extends. A single free nonterminal coset does not suffice for this guarantee.
\end{theorem}

The prescriptions are made on entire cosets. The terminal vertices inside $H$ are chosen by the construction, and the two free cosets may change the number of cosets on which a given generator is used.

The arithmetic content of the criterion is a classification of Hamiltonian cuts. For $D\in\Zq$ and $1\le s\le q$, define the partial map
\begin{equation}\label{eq:cut}
 R_{D,s}(0)=\bot,\qquad
 R_{D,s}(z)=
 \begin{cases}
 z+D,&1\le z<s,\\
 z+D+1,&s\le z<q
 \end{cases}\pmod q.
\end{equation}
Its arcs form one directed path and possibly some disjoint directed cycles. Write
\[
 Z_q(D)=\{s\in\{1,\ldots,q\}:R_{D,s}\text{ is a Hamiltonian path}\}.
\]
For subsets of $\Z$, distance means the least absolute difference, measured on the integer line.

\begin{theorem}[Distances between Hamiltonian cut spectra]\label{thm:spectra}
Let $q\ge3$ and $A,C\in\Zq\setminus\{0,-1\}$. If $q$ is even, then
\[
 \dist(Z_q(A),Z_q(C))=
 \begin{cases}
 0,&A\equiv C\pmod2,\\
 1,&A\not\equiv C\pmod2.
 \end{cases}
\]
If $q$ is odd and neither coefficient is $c=(q-1)/2$, then the spectra intersect. For the central coefficient,
\begin{equation}\label{eq:central-distance}
 \dist(Z_q(c),Z_q(D))
 =\min\{\gcd(q,D),\gcd(q,D+1)\}-1.
\end{equation}
For every $D\in\Zq$, the degenerate distances are
\begin{equation}\label{eq:degenerate-distance}
 \dist(Z_q(0),Z_q(D))=\gcd(q,D+1)-1,\qquad
 \dist(Z_q(-1),Z_q(D))=\gcd(q,D)-1.
\end{equation}
\end{theorem}

The extension criterion and the spectral compatibility problem express the same obstruction in different coordinates. Section~\ref{sec:reduction} proves that a fixed $E$ extends exactly when $Z_q(u_E)$ and $Z_q(u_E-B)$ are at distance at most one. Conversely, every triple $(q,B,u)$ is realized by such an extension problem on an abelian group. Choose $g>q$ with $g\equiv B\pmod q$, take $a=(0,1)$ and $b=(1,1)$ in $\Z_q\times\Z_g$, and choose $|E|\equiv u\pmod q$. Then $H=\Z_q\times\{0\}$, $r=1$, and $t=0$. Thus prescribed extension already contains the full compatibility problem in the abelian case. The elementary endpoint formulas determine the larger exceptional distances. The special case $C=-A-1$, where $Z_q(C)=q+1-Z_q(A)$, is the cut-reflection theorem used for the abelian case in~\cite{Park2026}.

\subsection*{The geometric argument}
The classical lattice description of $Z_q(D)$ can be read from cycle counts. Closing the open component in~\eqref{eq:cut} gives a permutation $Q_{D,s}$. The number
\[
 j_D(s)=\cyc(Q_{D,s})-1
\]
of remaining closed components traces a succession of triangular excursions. Each excursion corresponds to a primitive positive vector $U$ of
\[
 \Lambda_D=\{(x,y)\in\Z^2:Dx+(D+1)y\equiv0\pmod q\},
\]
and has height $\lfloor(q-1)/(x+y)\rfloor$. The zeros are precisely the Hamiltonian cuts. The endpoint divisors $\gcd(q,D+1)$ and $\gcd(q,D)$ determine the lengths of the initial and final linear segments of the profile, respectively.

Large excursions are separated by other excursions: positive integer combinations of two primitive vectors supply intervening lattice points. Our confinement theorem turns this separation into a comparison of two independently ordered sequences. Read their union from left to right and split it into maximal runs of points from the same sequence. Sector filling excludes five runs, and a length inequality excludes four. If the ranges overlap, a separated pair therefore has exactly three runs: one sequence lies entirely within a single gap of the other.

For actual cut spectra, the inner endpoint divisors make that containing gap too large to arise from a primitive vector of length at least three. The only remaining vector is $(1,1)$, which is present precisely at the odd central coefficient. Figure~\ref{fig:profiles} contrasts intersecting spectra with this unique exceptional geometry. The proof combines arc forcing with primitive lattice-point geometry; Appendix~\ref{app:profile} derives the cycle profile by local surgery on a torus.

\begin{figure}[tb]
\centering
\begin{minipage}[t]{0.48\textwidth}
\centering
\begin{tikzpicture}[x=0.178125cm,y=0.287500cm,font=\scriptsize,>=Stealth]
\draw[->] (1,0)--(36.6491,0) node[right] {$s$};
\draw[->] (1,0)--(1,9.2174) node[above] {$j(s)$};
\draw[thick] plot coordinates {(1,2) (2,1) (3,0) (4,1) (5,0) (6,1) (7,2) (8,1) (9,0) (10,1) (11,0) (12,1) (13,2) (14,1) (15,0) (16,1) (17,0) (18,1) (19,2) (20,1) (21,0) (22,1) (23,0) (24,1) (25,2) (26,1) (27,0) (28,1) (29,0) (30,1) (31,2) (32,1) (33,0)};
\draw[gray,densely dashed,thick] plot coordinates {(1,0) (2,1) (3,2) (4,3) (5,4) (6,5) (7,6) (8,7) (9,8) (10,7) (11,6) (12,5) (13,4) (14,3) (15,2) (16,1) (17,0) (18,1) (19,0) (20,1) (21,0) (22,1) (23,0) (24,1) (25,0) (26,1) (27,2) (28,1) (29,0) (30,1) (31,0) (32,1) (33,2)};
\fill (3,0) circle (1.6pt);
\fill (5,0) circle (1.6pt);
\fill (9,0) circle (1.6pt);
\fill (11,0) circle (1.6pt);
\fill (15,0) circle (1.6pt);
\fill (17,0) circle (1.6pt);
\fill (21,0) circle (1.6pt);
\fill (23,0) circle (1.6pt);
\fill (27,0) circle (1.6pt);
\fill (29,0) circle (1.6pt);
\fill (33,0) circle (1.6pt);
\draw[gray] (1,0) circle (2.5pt);
\draw[gray] (17,0) circle (2.5pt);
\draw[gray] (19,0) circle (2.5pt);
\draw[gray] (21,0) circle (2.5pt);
\draw[gray] (23,0) circle (2.5pt);
\draw[gray] (25,0) circle (2.5pt);
\draw[gray] (29,0) circle (2.5pt);
\draw[gray] (31,0) circle (2.5pt);
\node[below=5pt] at (1,0) {1};
\node[below=5pt] at (17,0) {17};
\node[below=5pt] at (33,0) {33};
\end{tikzpicture}
\par\smallskip {\small (a) Overlapping cut spectra}
\par{\scriptsize $q=33$; solid: $D=2$, dashed: $D=24$.}
\end{minipage}
\hfill
\begin{minipage}[t]{0.48\textwidth}
\centering
\begin{tikzpicture}[x=0.407143cm,y=0.328571cm,font=\scriptsize,>=Stealth]
\draw[->] (1,0)--(16.5965,0) node[right] {$s$};
\draw[->] (1,0)--(1,8.0652) node[above] {$j(s)$};
\draw[thick] plot coordinates {(1,0) (2,1) (3,2) (4,3) (5,4) (6,5) (7,6) (8,7) (9,6) (10,5) (11,4) (12,3) (13,2) (14,1) (15,0)};
\draw[gray,densely dashed,thick] plot coordinates {(1,2) (2,1) (3,0) (4,1) (5,0) (6,1) (7,0) (8,1) (9,0) (10,1) (11,0) (12,1) (13,2) (14,3) (15,4)};
\fill (1,0) circle (1.6pt);
\fill (15,0) circle (1.6pt);
\draw[gray] (3,0) circle (2.5pt);
\draw[gray] (5,0) circle (2.5pt);
\draw[gray] (7,0) circle (2.5pt);
\draw[gray] (9,0) circle (2.5pt);
\draw[gray] (11,0) circle (2.5pt);
\node[below=5pt] at (1,0) {1};
\node[below=5pt] at (8,0) {8};
\node[below=5pt] at (15,0) {15};
\end{tikzpicture}
\par\smallskip {\small (b) The central containing gap}
\par{\scriptsize $q=15$; solid: $D=7$, dashed: $D=5$.}
\end{minipage}
\caption{Cycle profiles computed directly from $Q_{D,s}$. Filled and open markers are Hamiltonian cuts. In (a) the two spectra share the cut $17$. In (b), the central spectrum $\{1,15\}$ encloses $Z_{15}(5)$ within its single gap; their distance is two.}\label{fig:profiles}
\end{figure}
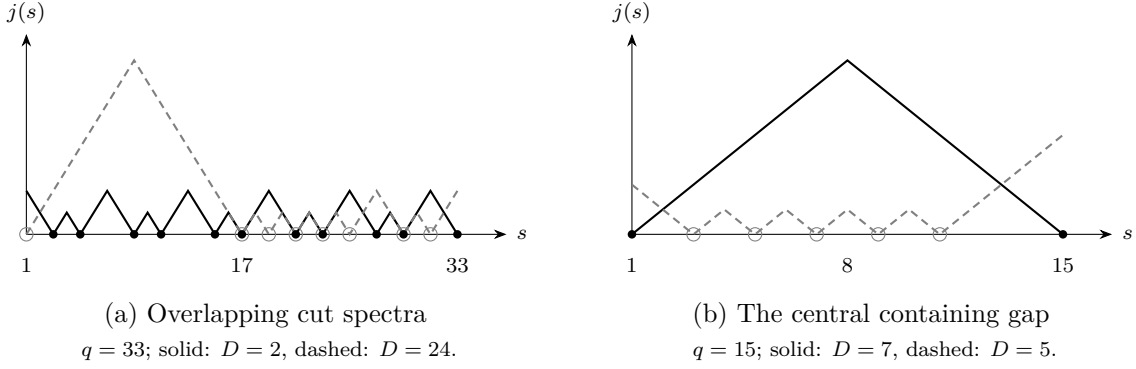

\subsection*{Relation to earlier work}
Darijani, Miraftab, and Witte Morris~\cite{DMM} proved that the Cartesian product of two directed cycles has two arc-disjoint Hamiltonian paths, answering a question of Gallian, and obtained the same conclusion for products of at least four directed cycles. They formulated the finite abelian two-generator conjecture and recorded the arc-forcing and lattice tools used to study it. The abelian conjecture was proved in~\cite{Park2026} by a self-reflection argument for Hamiltonian cuts together with the reductions of~\cite{DMM}. The confinement argument gives a uniform construction for normal difference subgroups, including the finite abelian case.

The primitive-ray parametrization and its stronger cycle-class interpretation are due to Curran and Witte~\cite{CW}. We use the consecutive-generator form and give a proof in Appendix~\ref{app:profile}. The local intersection argument follows their Lemma~7.9; a pairing of residues gives the total multiplicity and completes the profile. The principal result here is the exact classification of prescribed coset choices. Its arithmetic form is the complete distance formula of Theorem~\ref{thm:spectra}. The underlying confinement theorem compares independently ordered sequences through two geometric conditions: sector filling bounds alternation, and the length inequality excludes interlacing. The resulting exceptional residues yield the sharp extension theorem with two free cosets.

The normality hypothesis concerns the specified generating pair. It holds for every generating pair of an abelian group and makes $G$ metacyclic, since $H$ and $G/H$ are cyclic. Conversely, every finite nonabelian metacyclic group admits a generating pair satisfying the hypothesis; see Corollary~\ref{cor:metacyclic}. Without normality, connected two-generator Cayley digraphs need not have even one Hamiltonian path~\cite{Morris2013}. On a nilpotent group every two-generator Cayley digraph has one~\cite{Morris2012}, but Section~\ref{sec:examples} shows that the dihedral group of order eight, which is nilpotent, has a generating pair with a Hamiltonian path and no two arc-disjoint Hamiltonian paths.

We first formulate the terminal-coset extension problem in Section~\ref{sec:reduction}. Section~\ref{sec:geometry} describes the cut profiles and their two geometric inequalities. The confinement theorem and the spectral classification occupy Section~\ref{sec:confinement}. Section~\ref{sec:failure} derives the exceptional residues and proves the two-free-coset theorem. The final section records construction bounds and examples, and Appendix~\ref{app:profile} proves the cycle profile.

\section{The terminal-coset extension problem}\label{sec:reduction}
We retain the right Cayley convention. The following observation applies to arbitrary difference subgroups.

\begin{lemma}[Arc forcing]\label{lem:forcing}
Let $P$ be a Hamiltonian path in $\Cay(G;\{a,b\})$, and write $\eta=ba^{-1}$. Every left coset $xH$ not containing the terminal vertex uses one generator throughout. If the terminal vertex is $\tau$, the choices at
\[
 \tau\eta,\tau\eta^2,\ldots,\tau\eta^{q-1}
\]
are a string of $a$-choices followed by a string of $b$-choices, with either string possibly empty.
\end{lemma}
\begin{proof}
The arcs $x\to xb$ and $x\eta\to x\eta a=xb$ have the same head. Hence a $b$-choice at $x$ prevents an $a$-choice at $x\eta$. On a coset with all outgoing arcs present, a single $b$-choice propagates around the entire $\eta$-orbit. On the terminal coset, deleting its unique terminal position leaves the asserted ordered string.
\end{proof}
This is the arc-forcing observation underlying the classical terminal-coset theory; see~\cite[Lemma~3.9]{DMM} for its abelian formulation and historical attribution.

\begin{lemma}[Terminal cosets]\label{lem:terminal-cosets}
If two arc-disjoint Hamiltonian paths have their terminal vertices in distinct left cosets of $H$, their two unused arcs can be added to give a decomposition into two Hamiltonian cycles. In particular, the existence of a pair implies the existence of a pair whose terminal vertices lie in the same left coset of $H$.
\end{lemma}
\begin{proof}
On the terminal coset of the first path, the second path makes a complete constant generator choice. The first path therefore uses the opposite generator at every position except its terminal vertex. Add the opposite generator's arc there, and perform the analogous operation for the other path. Both colors now make complementary constant choices on every coset. Each is a permutation: the two possible predecessors of any head lie in the same $\eta$-orbit, and precisely one is chosen. The arc added to each spanning path closes that path into a Hamiltonian cycle. Deleting both cycles' outgoing arcs at a common vertex gives the last assertion.
\end{proof}

Assume from now on that $H$ is normal. Put $q=|H|$, $g=[G:H]$, and use~\eqref{eq:intro-coordinates}. In coordinates $\eta^z a^j$, right multiplication is
\begin{equation}\label{eq:arcs}
 (z,j)a=(z+c_j,j+1),\qquad
 (z,j)b=(z+c_j+r^j,j+1),
 \qquad
 c_j=\begin{cases}t,&j=g-1,\\0,&j<g-1,\end{cases}
\end{equation}
where the first coordinate is modulo $q$ and the second modulo $g$.

Let $g\ge2$ and fix $E\subseteq\{1,\ldots,g-1\}$. On the nonterminal cosets the two colors make the choices described in the introduction. Define
\begin{equation}\label{eq:returns}
 A=u_E=t+\sum_{j\in E}r^j,\qquad
 B=2t+\sum_{j=0}^{g-1}r^j,\qquad C=A-B\pmod q.
\end{equation}
We index the first and second colors by $i=0,1$. On the terminal coset let $\chi_i(z)$ be $0$ or $1$ according as color $i$ uses $a$ or $b$, and leave it undefined at that color's terminal vertex. The two return maps are
\begin{equation}\label{eq:return-map}
 z\longmapsto z+A+\chi_0(z),\qquad
 z\longmapsto z+(B-1-A)+\chi_1(z).
\end{equation}
Every nonterminal coset map is a bijection. Hence each color's nonterminal arcs decompose the vertices into $q$ disjoint directed segments from $Ha$ to $H$. The choices in $H$ join those segments.

\begin{lemma}[Realization of adjacent cuts]\label{lem:realization}
If $s\in Z_q(A)$ and $v\in Z_q(C)$ satisfy $|s-v|\le1$, the prescribed nonterminal choices extend to two arc-disjoint Hamiltonian paths with terminal vertices in $H$.
\end{lemma}
\begin{proof}
Color $0$ omits the outgoing arc at $0$, uses $a$ at $1\le z<s$, and uses $b$ at $s\le z<q$. For color $1$ use the table below, omitting the outgoing arc at $h$ and using $a$ at all remaining coordinates outside the indicated interval.
\begin{center}
\begin{tabular}{@{}ccc@{}}
\toprule
Relation & Terminal coordinate $h$ & Coordinates using $b$ in color $1$\\
\midrule
$v=s$ & $s\pmod q$ & $1,\ldots,s-1$\\
$v=s+1$ & $s$ & $0,\ldots,s-1$\\
$v=s-1$ & $s-1$ & $1,\ldots,s-2$\\
\bottomrule
\end{tabular}
\end{center}
Empty intervals are interpreted as empty sets. At every coordinate where both colors have an outgoing arc, their generators differ.

The first map in~\eqref{eq:return-map} is $R_{A,s}$. Under the reflection $x=h-z$, the second map becomes
\[
 x\longmapsto x+C+1-\chi_1(h-x),
\]
which is $R_{C,v}$ in each row. Thus each partial return joins all $q$ segments into one path. Exactly one vertex of $Ha$ has no incoming terminal arc in each color; it is the initial vertex of that path. Each resulting path visits every vertex of $G$, and their arc sets are disjoint.
\end{proof}

\begin{proposition}[Exact reduction for fixed choices]\label{prop:fixed-reduction}
For a fixed set $E$ of nonterminal coset choices, an extension with both terminal vertices in $H$ exists if and only if
\begin{equation}\label{eq:fixed-reduction}
 \dist(Z_q(u_E),Z_q(u_E-B))\le1.
\end{equation}
\end{proposition}
\begin{proof}
Sufficiency is Lemma~\ref{lem:realization}. For necessity, left-translate by the inverse of the first terminal vertex, an element of $H$. This preserves $E$ and places that terminal at $0$. Arc forcing gives a unique cut $s$ for color $0$.

Let $h$ be the second terminal coordinate. At every $z\ne0,h$, the second color uses the opposite generator. When $1<s<q$, its choices at $s-1,s$ would be $b,a$, causing an incoming conflict. Thus $h\in\{s-1,s\}$. The choice at zero is free. After reflection about $h$, the four possibilities $(h,\chi_1(0),v)$ are
\[
 (s-1,0,s-1),\quad(s-1,1,s),\quad
 (s,0,s),\quad(s,1,s+1).
\]
The reflected coefficient is always $C=u_E-B$.

When $s=1$, the other color either uses $a$ at every nonterminal position, giving $v=1$, or uses $b$ only at zero and has terminal coordinate one, giving $v=2$. When $s=q$, the symmetric possibilities give $v=q$ or $q-1$. The constant-choice cases allow arbitrary terminal coordinates, with the same reflected cut. This includes $q=2$ and completes the classification.
\end{proof}

The necessity in this proposition holds for each prescribed $E$. It is therefore stronger than a reduction that merely produces some convenient external assignment. Together with Lemma~\ref{lem:terminal-cosets}, it also gives a complete existence normal form when the nonterminal choices are allowed to vary.

\section{Cut profiles and their geometric inequalities}\label{sec:geometry}
Permutation products act from right to left. For $1\le s\le q$, let $\tau_s$ be the backward cycle on $0,\ldots,s-1$, extended by the identity, and let $T_v(z)=z+v$.

\subsection{Elementary cut identities}
The map $R_{D,s}$ is injective on $\Zq\setminus\{0\}$ and has unique missing image $D+s$. Its formal closure is
\begin{equation}\label{eq:Q}
 Q_{D,s}=T_{D+1}\tau_s.
\end{equation}
Deleting the formal arc $0\to D+s$ opens the cycle containing zero. Consequently,
\begin{equation}\label{eq:cycle-test}
 s\in Z_q(D)\quad\Longleftrightarrow\quad\cyc(Q_{D,s})=1.
\end{equation}
The permutation $Q_{D,s}$ records the cycle structure of the partial map, whose arcs belong to the two-generator digraph.

\begin{lemma}\label{lem:endpoints}
For $q\ge2$ and $D\in\Zq$, write $d_0=\gcd(q,D)$ and $d_1=\gcd(q,D+1)$. Then
\begin{align}
 \min Z_q(D)&=d_1,&\max Z_q(D)&=q+1-d_0,\label{eq:endpoints}\\
 Z_q(-D-1)&=q+1-Z_q(D),&
 s&\equiv d_1\pmod2\quad(s\in Z_q(D)).\label{eq:parity}
\end{align}
In particular, $Z_q(0)=\{1\}$ and $Z_q(-1)=\{q\}$.
\end{lemma}
\begin{proof}
The translation $T_{D+1}$ has $d_1$ cycles. If $s<d_1$, the support of $\tau_s$ misses at least one cycle. At $s=d_1$ it meets each cycle once, and its cyclic rejoining makes one cycle. This proves the minimum. Reflection by $z\mapsto-z$ conjugates $R_{D,s}$ to $R_{-D-1,q+1-s}$, proving the reflection identity and the maximum. Finally,
\[
 \operatorname{sgn}(Q_{D,s})=(-1)^{q-d_1+s-1},
\]
which equals the sign $(-1)^{q-1}$ of a $q$-cycle exactly when the stated parity condition holds.
\end{proof}

\subsection{Primitive rays and the cycle profile}
For the remainder of this section assume $q\ge3$ and $D\notin\{0,-1\}$. Put $N=q-1$, and consider
\begin{equation}\label{eq:lattice}
 \Lambda_D=\{(x,y)\in\Z^2:Dx+(D+1)y\equiv0\pmod q\},\qquad
 \len(x,y)=x+y.
\end{equation}
Primitivity is taken in $\Lambda_D$. Order the primitive rays in the first quadrant from the $y$-axis to the $x$-axis, retaining an internal ray precisely when its primitive vector has length at most $N$. Write the internal vectors as $U_1,\ldots,U_k$ and put
\[
 H_i=\left\lfloor\frac{N}{\len(U_i)}\right\rfloor,
 \qquad L=d_1-1,\qquad R=d_0-1.
\]

\begin{proposition}[Cycle profile; Curran--Witte]\label{prop:profile}
Define
\[
 z_0=L,\qquad z_i=L+2\sum_{j=1}^iH_j\quad(1\le i\le k).
\]
Then $z_k=N-R$, and, with $x=s-1$,
\begin{equation}\label{eq:profile}
 j_D(s):=\cyc(Q_{D,s})-1=
 \begin{cases}
 L-x,&0\le x\le L,\\
 \min\{x-z_{i-1},z_i-x\},&z_{i-1}\le x\le z_i,\\
 x-(N-R),&N-R\le x\le N.
 \end{cases}
\end{equation}
In an internal excursion, every closed component of $R_{D,s}$ has primitive displacement $U_i$ and length $\len(U_i)$. Moreover,
\begin{equation}\label{eq:profile-zeros}
 Z_q(D)-1=\{z_0,\ldots,z_k\},\qquad
 j_D(s)=\dist(s,Z_q(D)).
\end{equation}
\end{proposition}

This is the consecutive-generator specialization of the cycle-class theorem in~\cite[Theorem~7.2 and Proposition~4.12]{CW}; its zero set is the Hamiltonian-use-count parametrization recalled in~\cite[Section~3B]{DMM}. Appendix~\ref{app:profile} proves the proposition from cut surgery, a local torus intersection, and a residue count. Thus the lattice parametrization is established within the present paper. Figure~\ref{fig:lattice} illustrates the proposition.
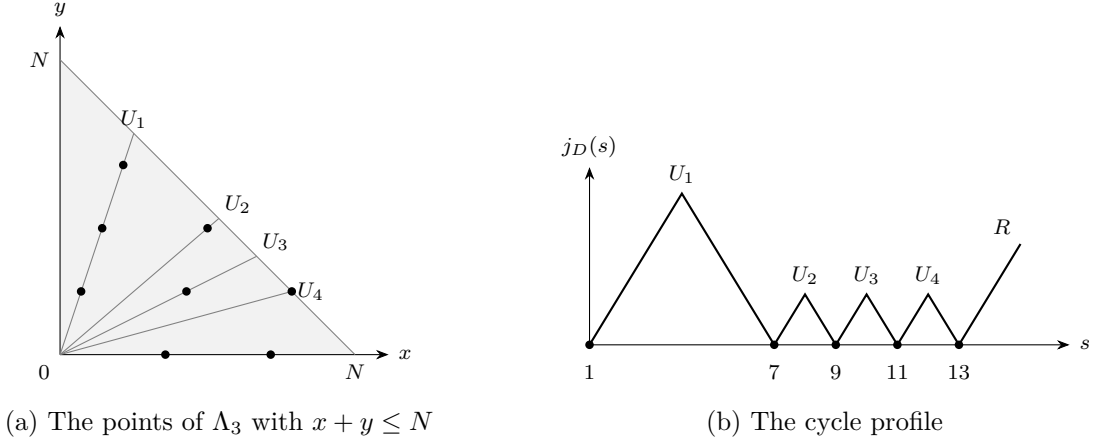
\begin{figure}[tb]
\centering
\begin{minipage}[t]{0.40\textwidth}
\centering
\begin{tikzpicture}[x=0.2786cm,y=0.2786cm,font=\scriptsize,>=Stealth]
\fill[black!5] (0,0)--(14,0)--(0,14)--cycle;
\draw[->] (0,0)--(15.6,0) node[right] {$x$};
\draw[->] (0,0)--(0,15.6) node[above] {$y$};
\draw[gray] (14,0)--(0,14);
\node[below] at (14,0) {$N$};
\node[left] at (0,14) {$N$};
\draw[gray] (0,0)--(3.500,10.500);
\node[above=1pt,inner sep=1pt] at (3.500,10.500) {$U_1$};
\draw[gray] (0,0)--(7.538,6.462);
\node[above right=1pt,inner sep=1pt] at (7.538,6.462) {$U_2$};
\draw[gray] (0,0)--(9.333,4.667);
\node[above right=1pt,inner sep=1pt] at (9.333,4.667) {$U_3$};
\draw[gray] (0,0)--(11.000,3.000);
\node[right=1pt,inner sep=1pt] at (11.000,3.000) {$U_4$};
\foreach \p in {(1,3),(2,6),(3,9),(7,6),(6,3),(11,3),(5,0),(10,0)} \fill \p circle (1.6pt);
\node[below left] at (0,0) {$0$};
\end{tikzpicture}
\par\smallskip {\small (a) The points of $\Lambda_{3}$ with $x+y\le N$}
\end{minipage}
\hfill
\begin{minipage}[t]{0.56\textwidth}
\centering
\begin{tikzpicture}[x=0.4071cm,y=0.6667cm,font=\scriptsize,>=Stealth]
\draw[->] (1,0)--(16.596,0) node[right] {$s$};
\draw[->] (1,0)--(1,3.525) node[above] {$j_D(s)$};
\draw[thick] plot coordinates {(1,0) (2,1) (3,2) (4,3) (5,2) (6,1) (7,0) (8,1) (9,0) (10,1) (11,0) (12,1) (13,0) (14,1) (15,2)};
\fill (1,0) circle (1.6pt);
\node[below=4pt] at (1,0) {1};
\fill (7,0) circle (1.6pt);
\node[below=4pt] at (7,0) {7};
\fill (9,0) circle (1.6pt);
\node[below=4pt] at (9,0) {9};
\fill (11,0) circle (1.6pt);
\node[below=4pt] at (11,0) {11};
\fill (13,0) circle (1.6pt);
\node[below=4pt] at (13,0) {13};
\node[above] at (4.0,3) {$U_1$};
\node[above] at (8.0,1) {$U_2$};
\node[above] at (10.0,1) {$U_3$};
\node[above] at (12.0,1) {$U_4$};
\node[above left] at (15,2) {$R$};
\end{tikzpicture}
\par\smallskip {\small (b) The cycle profile}
\end{minipage}
\caption{Proposition~\ref{prop:profile} for $q=15$ and $D=3$. The internal rays $U_1,\ldots,U_{4}$ carry $3,1,1,1$ lattice points, the positive $x$-axis carries $R=2$, and the positive $y$-axis carries $L=0$. Each internal ray contributes one triangular excursion of $j_D$ whose height is its multiplicity; the zeros are $Z_{15}(3)=\{1,7,9,11,13\}$.}\label{fig:lattice}
\end{figure}

To compare conventions: after translation, the classical standard path with $d$ uses of the second generator is our cut path with $s=q-d$. If $m=\ord_{\Zq}(D)$, $n=q/m=d_0$, and $n(D+1)=eD$, the classical lattice coordinates map to~\eqref{eq:lattice} by
\[
 (X,Y)\longmapsto(mX-eY,nY).
\]
Its determinant is $mn=q$, and its length form becomes $x+y$. Reversing the use count reverses the ray order. The $y$- and $x$-axis multiplicities are therefore $L$ and $R$, respectively.

The distinguished cycle of $Q_{D,s}$ is the one containing zero. Its length in the $i$th excursion is $q-j_D(s)\len(U_i)$. At a peak, the closed cycles occupy as many vertices as possible, leaving at least the vertex $0$ to the distinguished cycle. For example, when $q=9,D=2,s=7$,
\[
 Q_{2,7}=(0)(1\ 3\ 5\ 7)(2\ 4\ 6\ 8),
\]
so the distinguished cycle has length one and the other two have length four.
Moving toward a nearest zero in~\eqref{eq:profile} absorbs one closed component at each step.

\subsection{Sector filling and the length inequality}
The ordered weights of the cut sequence $Z_q(D)-1$ are the ray multiplicities
\[
 L,H_1,\ldots,H_k,R,
\]
listed from the $y$-axis to the $x$-axis.
The multiplicity between two weights is the sum of the weights strictly between them.

\begin{lemma}[Sector filling]\label{lem:sector}
For two distinct rays of positive multiplicities $p,t$, the intervening multiplicity $M$ satisfies
\begin{equation}\label{eq:sector-exact}
 M\ge\Theta(p,t)
 =\frac{(p-1)(t-1)+\gcd(p,t)-1}{2}
 \ge\frac{(p-1)(t-1)}2.
\end{equation}
The assertion includes either or both boundary rays when their multiplicities are positive.
\end{lemma}
\begin{proof}
Let $U,V$ be the primitive vectors. Their lengths are at most $N/p,N/t$. Every positive pair $i,j$ with $i/p+j/t\le1$ supplies a distinct lattice point $iU+jV$ strictly between the rays and of length at most $N$. The intervening multiplicity counts all these points, including distinct multiples on the same primitive ray. The number of such pairs is $\Theta(p,t)$: pairing the interior points of $[0,p]\times[0,t]$ across its center leaves $\gcd(p,t)-1$ points on the diagonal $i/p+j/t=1$ and gives the formula.
\end{proof}

\begin{lemma}[Length inequality]\label{lem:length}
Every internal weight $H$ of $Z_q(D)-1$ satisfies
\begin{equation}\label{eq:length}
 (L+R+2)H\le N.
\end{equation}
\end{lemma}
\begin{proof}
For $(x,y)\in\Lambda_D$ with $x,y>0$, reducing the defining congruence modulo $d_1$ and $d_0$ gives $d_1\mid x$ and $d_0\mid y$. Hence
\[
 \len(x,y)\ge d_0+d_1=L+R+2.
\]
Now use $H\len(U)\le N$.
\end{proof}

The two estimates have complementary meanings. Sector filling separates large excursions from one another. The length inequality limits each excursion by the sum of the two endpoint divisors. The comparison argument below combines the quadratic lower bound in~\eqref{eq:sector-exact} with~\eqref{eq:length}.

\section{Confinement of separated sequences}\label{sec:confinement}
We state the geometric result independently of its lattice realization. This separates the combinatorial comparison from the arithmetic restriction on primitive lengths.

\begin{definition}\label{def:sector-filled}
Let $S=\{s_0<\cdots<s_k\}\subseteq\{0,\ldots,N\}$, with $k\ge1$, and suppose its gaps have positive integer half-lengths $h_i=(s_i-s_{i-1})/2$. Its boundary weights are $L=s_0$ and $R=N-s_k$; its ordered weights are
\[
 L,h_1,\ldots,h_k,R.
\]
We call $S$ \emph{sector-filled} if any two distinct positions carrying positive weights $p,t$ satisfy
\begin{equation}\label{eq:weak-sector}
 M\ge\frac{(p-1)(t-1)}2,
\end{equation}
where $M$ is the sum of the intervening weights. We say that it satisfies the length inequality if
\begin{equation}\label{eq:abstract-length}
 (L+R+2)h_i\le N\qquad(1\le i\le k).
\end{equation}
\end{definition}

The definition isolates the quadratic consequence of sector filling. Since $M$ is integral,~\eqref{eq:weak-sector} includes rounding up. Lemmas~\ref{lem:sector} and~\ref{lem:length} show that the normalized cut spectra satisfy both conditions.

\subsection{Sector filling bounds the number of alternations}
Suppose $S,T$ are separated by $\delta=\dist(S,T)\ge2$. Label the elements of their ordered union by their set of origin. A \emph{run} is a maximal block of consecutive elements with the same label. Any internal gap containing a point of the other sequence has half-length at least $\delta$.

\begin{lemma}[Five runs]\label{lem:five-runs}
The ordered union of two separated sector-filled sequences has at most four runs.
\end{lemma}
\begin{proof}
Suppose that five consecutive runs have labels
\[
 S_0,\quad T_0,\quad S_1,\quad T_1,\quad S_2.
\]
The $S$-gaps containing $T_0,T_1$ have respective half-lengths $\alpha,\gamma$. The $T$-gap containing $S_1$ has half-length $\beta$. Reversing the order if necessary, which interchanges $T_0$ and $T_1$, assume $\alpha\le\gamma$.

Let $M$ be the multiplicity between the $\alpha$- and $\gamma$-gaps of $S$. The first and last points of $S_1$ span length $2M$. The two ends of the $\beta$-gap lie at least $\delta$ outside those points. Hence
\begin{equation}\label{eq:beta-lower}
 \beta\ge M+\delta
 \ge\frac{(\alpha-1)(\gamma-1)}2+\delta
 \ge\frac{(\alpha-1)^2}2+\delta.
\end{equation}

In the $T$-sequence, let $p$ be the weight immediately preceding the run $T_0$: it is its preceding gap weight, or the left boundary weight if $T_0$ is the first $T$-run. In the former case that gap contains $S_0$, so $p\ge\delta$. In the latter case $\min T\ge\delta$, again giving $p\ge\delta$. The multiplicity between $p$ and the $\beta$-gap is half the span of $T_0$, at most $\alpha-\delta$. Thus
\[
 \frac{\beta-1}{2}
 \le\frac{(p-1)(\beta-1)}2
 \le\alpha-\delta,
\]
so $\beta\le2\alpha-2\delta+1$. Combining this with~\eqref{eq:beta-lower} yields
\[
 (\alpha-3)^2+6\delta-10\le0.
\]
Its left side is at least two, a contradiction.
\end{proof}

\begin{figure}[tb]
\centering
\begin{tikzpicture}[x=1cm,y=1cm,>=Stealth,font=\small]
 \draw (0,0)--(12.1,0);
 \foreach \x/\lab in {0.8/$S_0$,3.2/$T_0$,5.7/$S_1$,8.2/$T_1$,10.7/$S_2$}{
  \draw[very thick] (\x-0.55,0)--(\x+0.55,0);
  \node[below=5pt] at (\x,0) {\lab};
 }
 \draw[decorate,decoration={brace,amplitude=4pt}] (1.35,0.35)--(5.15,0.35) node[midway,above=5pt] {$2\alpha$};
 \draw[decorate,decoration={brace,amplitude=4pt}] (6.25,0.35)--(10.15,0.35) node[midway,above=5pt] {$2\gamma$};
 \draw[decorate,decoration={brace,mirror,amplitude=4pt}] (3.75,-0.65)--(7.65,-0.65) node[midway,below=5pt] {$2\beta$};
 \draw[densely dashed] (5.15,0.05)--(5.15,0.92);
 \draw[densely dashed] (6.25,0.05)--(6.25,0.92);
 \node[above] at (5.7,0.9) {$2M$};
\end{tikzpicture}
\caption{Five alternating runs. The $\beta$-gap contains the middle $S$-run. Sector filling between the $\alpha$- and $\gamma$-gaps makes $\beta$ large, while sector filling on the left in $T$ gives the incompatible upper bound. Distances between runs are schematic.}\label{fig:five-runs}
\end{figure}
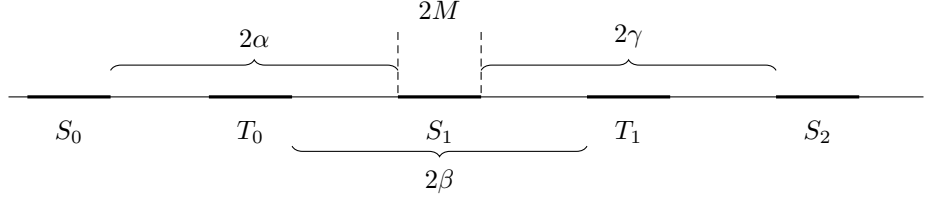

An immediate consequence already covers nested ranges: if $\min S<\min T\le\max T<\max S$, their merged order starts and ends in $S$. It has an odd number of runs, so Lemma~\ref{lem:five-runs} leaves exactly three. Thus $T$ lies in a single $S$-gap. For nested hulls, sector filling alone yields confinement.

\subsection{The length inequality excludes four runs}
The length inequality is preserved when the ambient interval is restricted to the convex hull of the two sequences. Indeed, if $c$ is removed at the left and $d$ at the right, replace
\[
 N\text{ by }N'=N-c-d,\qquad
 L,R\text{ by }L-c,R-d.
\]
Each boundary weight decreases, preserving~\eqref{eq:weak-sector}. For an unchanged internal weight $h$,
\begin{equation}\label{eq:hull-restriction}
 N'-(L+R-c-d+2)h
 =N-(L+R+2)h+(c+d)(h-1)\ge0.
\end{equation}

\begin{lemma}[Four runs]\label{lem:four-runs}
Two separated sector-filled sequences satisfying the length inequality cannot have exactly four runs.
\end{lemma}
\begin{proof}
Restrict to the convex hull of $S\cup T$ as above and rename its length $N$. Exchange $S,T$ if necessary so the runs are $S_0,T_0,S_1,T_1$. Then
\[
 \min S=0,\quad \max S=N-r,\qquad
 \min T=\ell,\quad \max T=N,
\]
where $\ell,r\ge\delta\ge2$. Let $[x,y]$ be the $S$-gap containing $T_0$, of half-length $\alpha$, and $[z,w]$ the $T$-gap containing $S_1$, of half-length $\beta$. Both half-lengths are at least $\delta$.

At the three transitions between runs, separation gives
\[
 x\le\ell-\delta,\qquad y-z\ge\delta,\qquad N-w\le r-\delta.
\]
Since $y=x+2\alpha$ and $w=z+2\beta$,
\begin{equation}\label{eq:four-upper}
 N=x+2\alpha+2\beta-(y-z)+(N-w)
 \le\ell+r+2\alpha+2\beta-3\delta.
\end{equation}
On the other hand, the two length inequalities imply
\begin{align*}
 2N&\ge(r+2)\alpha+(\ell+2)\beta\\
   &=4(\alpha+\beta)+(r-2)\alpha+(\ell-2)\beta\\
   &\ge4(\alpha+\beta)+2(\ell+r)-8.
\end{align*}
Twice~\eqref{eq:four-upper} contradicts this, because $6\delta-8>0$.
\end{proof}

\begin{theorem}[Confinement]\label{thm:confinement}
Let $S,T\subseteq\{0,\ldots,N\}$ be sector-filled sequences satisfying the length inequality. Suppose their convex hulls overlap. If $\delta=\dist(S,T)\ge2$, then, after exchanging the sequences if necessary, all of $T$ lies in one gap of $S$. More precisely, for two consecutive points $x,x+2H$ of $S$,
\begin{equation}\label{eq:confinement}
 x+\delta\le\min T\le\max T\le x+2H-\delta.
\end{equation}
\end{theorem}
\begin{proof}
Overlapping hulls and separation give at least three runs. Lemma~\ref{lem:five-runs} gives at most four, and Lemma~\ref{lem:four-runs} excludes four. With three runs, the middle run is the whole of one sequence and is enclosed by consecutive points of the other. Separation gives~\eqref{eq:confinement}.
\end{proof}

\begin{remark}\label{rem:length-needed}
The two hypotheses rule out different configurations. Sector filling alone forces a separated nested sequence into one gap. The length inequality excludes interlacing. For example, at $N=10$ the sequences
\[
 S=\{0,6,8\},\qquad T=\{2,4,10\}
\]
satisfy even the exact sector bound~\eqref{eq:sector-exact}, but have distance two and four alternating runs. Their largest weight is three and their boundary sum is two, so $(L+R+2)h=12>N$. Thus the length inequality distinguishes this abstract interlacing from a cut spectrum.
\end{remark}

\subsection{Proof of the spectral classification}
\begin{proof}[Proof of Theorem~\ref{thm:spectra}]
Assume first that $A,C$ are nondegenerate and that their spectral distance is $\delta\ge2$. Their endpoint divisors are proper divisors of $q$, so the convex hulls of $Z_q(A)-1$ and $Z_q(C)-1$ overlap. Apply Theorem~\ref{thm:confinement}; exchange $A,C$ so that the latter spectrum is contained in one gap of the former. Put
\[
 a=\gcd(q,A),\quad b=\gcd(q,A+1),\qquad
 c_0=\gcd(q,C),\quad d=\gcd(q,C+1).
\]
Comparing endpoints gives $c_0-a\ge\delta$ and $d-b\ge\delta$, hence $c_0,d\ge3$. Since $\gcd(c_0,d)=1$ and $c_0d\mid q$,
\begin{equation}\label{eq:divisor-sum}
 c_0+d\le\frac{c_0d}{3}+3\le\frac q3+3;
\end{equation}
this is $(c_0-3)(d-3)\ge0$.

Let $H$ be the height of the containing gap. Its width is at least
\[
 2H\ge q+1-c_0-d+2\delta\ge\frac{2q}{3}+2.
\]
Since $H=\lfloor(q-1)/\len(U)\rfloor$, the corresponding positive primitive vector has length less than three. It is therefore $(1,1)$, and its lattice congruence gives $2A+1=0\pmod q$. Consequently $q$ is odd and $A=(q-1)/2$.

For this coefficient the two steps are opposite units. A simple path in the resulting bidirected cycle cannot change direction, because that would revisit the preceding vertex. Hence $Z_q((q-1)/2)=\{1,q\}$. Lemma~\ref{lem:endpoints} gives~\eqref{eq:central-distance}. In every other nondegenerate case the distance is at most one, and~\eqref{eq:parity} gives the exact parity statement. Finally, the two degenerate spectra in Lemma~\ref{lem:endpoints} give~\eqref{eq:degenerate-distance}.
\end{proof}

The exceptional coefficient is characterized by the unique positive lattice direction of total length two. This also explains why a merely large internal multiplicity is insufficient: the containing gap must span the entire effective interval of another spectrum, and the endpoint divisors force its primitive length all the way down to two.

\section{Exceptional residues and prescribed extension}\label{sec:failure}
We turn the spectral classification into a criterion for the return residue $u_E$. Degenerate coefficients are allowed throughout this section.

\begin{proposition}\label{prop:failure}
For every $q\ge2$ and $B\in\Zq$,
\begin{equation}\label{eq:failure-identity}
 \F(q,B)=\{u\in\Zq:\dist(Z_q(u),Z_q(u-B))>1\},
\end{equation}
where $\F(q,B)$ is given by~\eqref{eq:Fminus}--\eqref{eq:Fcentral}.
\end{proposition}
\begin{proof}
The cases $u=0$ or $u-B=-1$ give~\eqref{eq:Fminus}; $u=-1$ or $u-B=0$ give~\eqref{eq:Fplus}, by the degenerate distance formulas. The only other possibilities are $u=c$ or $u-B=c$ for odd $q$. Multiplying the other coefficient and its successor by the unit two in~\eqref{eq:central-distance} gives the two conditions in~\eqref{eq:Fcentral}. They also handle any coincident entries of the displayed pairs. For $q=2$, the two singleton cut spectra are always at distance at most one.
\end{proof}

The bounds four and six are attained:
\[
 \F(12,5)=\{0,4,5,11\},\qquad
 \F(105,4)=\{0,3,4,52,56,104\}.
\]
The involution $u\mapsto B-1-u$ preserves the exceptional set. The sparseness relevant to extension is stronger than its cardinality: it contains no three consecutive residues.

\begin{lemma}[Three consecutive residues]\label{lem:three}
For all $q\ge2$, $B,u\in\Zq$, at least one of $u,u+1,u+2$ lies outside $\F(q,B)$.
\end{lemma}
\begin{proof}
For $q\le6$ the assertion follows directly from~\eqref{eq:Fminus}--\eqref{eq:Fcentral}. For instance, the nonempty sets at $q=6$ are $\{0\}$, $\{2,5\}$, $\{0,3\}$, and $\{5\}$.

For even $q\ge8$, the set is contained in the two consecutive pairs $\{-1,0\}$ and $\{B-1,B\}$. A triple of consecutive residues in this union forces $B\in\{\pm1,\pm2\}$. Substituting these values in the two gcd conditions leaves at most two exceptional residues.

Now let $q\ge7$ be odd and write $h=2^{-1}$ in $\Zq$. The candidate set is
\[
 \{-1,-h,0\}+\{0,B\}.
\]
After multiplication by two, it becomes $I\cup(I+2B)$ with $I=\{-2,-1,0\}$. A three-term progression of step two in this union has two terms in one of the two translates of $I$. Their difference cannot be four when $q\ge7$; they must be adjacent terms of the progression and the endpoints of that translate. Placing the remaining term in the other translate forces
\[
 B\in\{\pm1,\pm2,\pm3/2\}.
\]
For $B=\pm1$ the exceptional set has one point. For $B=3/2$ it is contained in $\{-1,B\}$, and for $B=-3/2$ in $\{0,B-1\}$. If $B=2$, the pair~\eqref{eq:Fminus} is absent. Neither $0$ nor $1$ can occur from the other two pairs when $q\ge7$. Thus neither endpoint pair $\{-1,0\}$ or $\{1,2\}$ required above is present. The case $B=-2$ follows by the reflection $u\mapsto-u-1$, which sends $\F(q,B)$ to $\F(q,-B)$.
\end{proof}

\begin{lemma}[Units fixing $B$]\label{lem:B-units}
Suppose $w,z$ are units modulo $q$ with $wB=zB=B$. Then for every $u$, at least one of
\begin{equation}\label{eq:three-states}
 u,\qquad u+w,\qquad u+w+z
\end{equation}
lies outside $\F(q,B)$.
\end{lemma}
\begin{proof}
Let
\[
 q_0=\prod_{p^e\parallel q,\ p\nmid B}p^e.
\]
Every common divisor of $q$ with $B-1,B+1,2B-1$, or $2B+1$ divides $q_0$. Thus each exceptional residue modulo $q$ projects to an exceptional residue modulo $q_0$, by Proposition~\ref{prop:failure}. In the odd case the central residue projects to $(q_0-1)/2$. If $q_0\le2$, the original exceptional set is empty.

Otherwise $B$ is a unit modulo $q_0$, so $w\equiv z\equiv1\pmod{q_0}$. If all three residues in~\eqref{eq:three-states} were exceptional, their images would be three consecutive exceptional residues modulo $q_0$, contradicting Lemma~\ref{lem:three}.
\end{proof}

\subsection{Extension and existence}
\begin{proof}[Proof of Theorem~\ref{thm:extension}]
The exact criterion is Proposition~\ref{prop:fixed-reduction} together with Proposition~\ref{prop:failure}. For the two free cosets, let
\[
 u_0=t+\sum_{k\in E_0}r^k.
\]
The relations~\eqref{eq:intro-coordinates} give
\begin{equation}\label{eq:B-invariant}
 (r-1)B=0\pmod q.
\end{equation}
Hence $r^i$ and $r^j$ are units fixing $B$. The three prescribed extensions have return residues
\[
 u_0,\qquad u_0+r^i,\qquad u_0+r^i+r^j.
\]
Lemma~\ref{lem:B-units} supplies a successful residue, and Lemma~\ref{lem:realization} completes its paths.

For sharpness, take the consistent parameters
\[
 (q,g,r,t)=(12,9,1,4)
\]
and prescribe $a$ in the first color on every nonterminal coset except one. Here $B=5$, and the two possible residues are $4,5$, both in $\F(12,5)$. Proposition~\ref{prop:fixed-reduction} excludes every terminal completion of these choices. With two free cosets the three possible residues are $4,5,6$; the last succeeds with cuts $(1,2)$. This proves the stated optimality.
\end{proof}

\begin{proof}[Proof of Theorem~\ref{thm:existence}]
When $g\ge3$, apply Theorem~\ref{thm:extension}, prescribing $a$ on all other nonterminal cosets.

When $g=2$, choose $E=\varnothing$. The two return coefficients are $A=t$ and $C=-t-r-1$. The relations $r^2=1$ and $rt=t$ imply
\[
 r(t+1)=t+r.
\]
Thus $c=\gcd(q,t+1)=\gcd(q,t+r)$ is an initial successful cut for both coefficients by Lemma~\ref{lem:endpoints}. Lemma~\ref{lem:realization} applies with $(s,v)=(c,c)$, including $q=2$.

When $g=1$, the group is cyclic and $a=\eta^t$, $b=\eta^{t+1}$. The generators being nonidentity means $t\notin\{0,-1\}$, so $q\ge3$. Theorem~\ref{thm:spectra} applies to $t,-t-1$. If one is the odd central coefficient, both are, and they share the endpoint cuts. Otherwise the theorem supplies equal or adjacent cuts. Use the terminal table of Lemma~\ref{lem:realization} directly on this single copy of $\Zq$. Its two partial maps are the required paths.
\end{proof}

\begin{corollary}\label{cor:metacyclic}
Every finite nonabelian metacyclic group admits a generating pair $a,b$ of distinct nonidentity elements for which $\Cay(G;\{a,b\})$ has two arc-disjoint Hamiltonian paths.
\end{corollary}
\begin{proof}
Choose a cyclic normal subgroup $K=\langle\eta\rangle$ with cyclic quotient, and let $aK$ generate $G/K$. Put $b=\eta a$. Then $G=\langle a,b\rangle$ and $\langle ba^{-1}\rangle=K$. Noncommutativity ensures $1<K<G$, so $a,b$ are distinct and lie outside $K$. Apply Theorem~\ref{thm:existence}.
\end{proof}

The extension result also gives at least $2^{g-3}$ distinct ordered path pairs when $g\ge3$: prescribe the other $g-3$ cosets arbitrarily and complete each prescription. Distinct prescriptions differ on a full nonterminal coset of the first path.

\section{Construction, examples, and further directions}\label{sec:examples}
\subsection{Construction and sharpness}
For a prescribed set $E$, compute $B,u_E$ and the gcds in~\eqref{eq:Fminus}--\eqref{eq:Fcentral}. This decides extension directly. With two free cosets, inspect the three residues in Theorem~\ref{thm:extension}. Once a successful residue is chosen, test the two permutations $Q_{u,s},Q_{u-B,s}$ for $s=1,\ldots,q$. Following the orbit of zero tests each in $O(q)$ operations. A scan of the two sorted spectra finds compatible cuts, and the terminal table gives the full successor maps. The total construction takes $O(q^2+qg)$ arithmetic operations and $O(qg)$ storage. The arithmetic model uses residues modulo $q$ and the explicit graph size $qg$; integer factorization is not required by this construction.

In the sharpness example $(12,9,1,4)$, leave cosets $1,2$ free and prescribe $a$ elsewhere in the first color. The chosen set is $E=\{1,2\}$, and the cuts are $(1,2)$. With vertex $(z,j)$ numbered $qj+z$, the two paths have endpoint pairs
\[
 (13,0),\qquad(12,1),
\]
and the unused arcs are $(0,12)$ and $(1,13)$.

A nonabelian example has
\[
 (q,g,r,t)=(84,21,37,28),\qquad (i,j)=(3,18),
\]
with prescribed selected cosets
\[
 E_0=\{1,2,5,6,8,9,11,12,14,15,17,20\}.
\]
Here $B=77$, the three residues are $76,77,78$, and the first two fail while the third has cuts $(1,2)$. Each path has $1764$ vertices. These examples concern the designated terminal coset and fixed external choices; the unrestricted existence theorem continues to hold. Explicit vertex lists for these and four further examples, together with programs checking Proposition~\ref{prop:profile} and the results of Sections~\ref{sec:confinement} and~\ref{sec:failure} on finite ranges, are provided as supplementary material.

\subsection{The exceptional quadratic relation}
The central coefficient has a second elementary interpretation. Let a primitive $q$th root of unity act on two variables by
\[
 (X,Y)\longmapsto(\zeta^D X,\zeta^{D+1}Y).
\]
The nonnegative lattice points of $\Lambda_D$ are exactly the exponent vectors of invariant monomials. Among monomials involving both variables, total degree two has the unique candidate $XY$. It is invariant precisely when $2D+1=0\pmod q$, the same condition that creates the exceptional cut spectrum. Thus the spectral obstruction is the smallest possible positive relation between the two steps.

\subsection{Beyond a normal difference subgroup}
Arc forcing and Lemma~\ref{lem:terminal-cosets} hold without normality. To go beyond normality, one needs complementary external choices whose components all meet a terminal section. If their endpoint permutations are translations of $H$, the spectral theorem determines the remaining joining question. Constructing such external path families is an additional problem.

The obstruction already appears in the dihedral group of order eight,
\[
 G=\langle a,b:a^4=b^2=1,\ bab^{-1}=a^{-1}\rangle,
 \qquad H=\langle ab\rangle.
\]
The graph has a Hamiltonian path, for example the word $a^3ba^3$, but no two arc-disjoint Hamiltonian paths for this generating pair. To see this, normalize both terminals into $H$ using Lemma~\ref{lem:terminal-cosets}. On the three external cosets $aH,a^2H,a^3H$, arc forcing gives complementary constant choices. Equal first-color choices on either adjacent pair force a second-color or first-color $b$-cycle away from $H$. Hence the choices alternate. The color using $a$ on $a^2H$ and $b$ on its two neighbors then contains the closed cycle
\[
 a^2\longrightarrow a^3\longrightarrow a^3b
 \longrightarrow a^2b\longrightarrow a^2,
\]
again disjoint from $H$. This rules out the pair. In particular, Theorem~\ref{thm:existence} does not extend to all generating pairs of nilpotent groups. In the same group, the pair $(b,ab)$ has difference subgroup $\langle a\rangle$, which is normal, so its Cayley digraph does have two arc-disjoint Hamiltonian paths. The contrast shows how the generating pair determines the relevant difference subgroup.

\subsection{Quantitative confinement and several spectra}
The four-run argument also measures the effect of weakening the length inequality. Suppose two sector-filled sequences in $\{0,\ldots,N\}$ have overlapping hulls, satisfy
\[
 (L+R+2)h_i\le N+\varepsilon
\]
for the same $\varepsilon\ge0$, and are separated by $\delta\ge2$. If neither is contained in a gap of the other, then
\begin{equation}\label{eq:relaxed-confinement}
 \varepsilon\ge\delta(\delta-1).
\end{equation}
Indeed, restricting to their joint hull preserves the same allowance $\varepsilon$. Lemma~\ref{lem:five-runs} leaves four runs. In the notation of Lemma~\ref{lem:four-runs}, the relaxed length inequalities give
\[
 2N+2\varepsilon
 \ge4(\alpha+\beta)+\delta(\ell+r-4).
\]
Combining this with twice~\eqref{eq:four-upper} yields
\[
 2\varepsilon\ge(\delta-2)(\ell+r)+2\delta
 \ge2\delta(\delta-1).
\]
The example in Remark~\ref{rem:length-needed} attains equality at $\varepsilon=\delta=2$.

For several cut spectra, pairwise intersection leaves a further compatibility question. Direct evaluation of~\eqref{eq:Q} gives
\begin{align*}
 Z_{20}(4)&=\{5,7,9,13,15,17\},\\
 Z_{20}(6)&=\{1,13,15,17,19\},\\
 Z_{20}(10)&=\{1,3,5,7,9,11\}.
\end{align*}
These three sets intersect pairwise, while their total intersection is empty. A quantitative formulation is the least total number of closed components at a common cut,
\[
 \kappa(D_1,\ldots,D_m)
 =\min_{1\le s\le q}\sum_{i=1}^m j_{D_i}(s).
\]
In this example $\kappa(4,6,10)=2$, attained at $s=13$ with component counts $(0,0,2)$. Bounds for $\kappa$ would describe simultaneous cut choices. An application to higher-multiplicity path packing would also require a suitable terminal arrangement in a digraph of sufficient degree. A separate counting question is to determine how many prescribed coset choices extend when their cardinality is fixed.

\appendix
\section{A self-contained proof of the cycle profile}\label{app:profile}
We prove Proposition~\ref{prop:profile} in the consecutive-generator coordinates. This is a proof of the stated specialization of the Curran--Witte cycle-class theorem, with its geometric antecedents made explicit. The local argument is the one-crossing principle of~\cite[Lemma~7.9]{CW}. The final multiplicity count is obtained by pairing residues.

\subsection{Holding the initial vertex fixed}
Fix $q\ge3$, $D\in\{1,\ldots,q-2\}$, and $N=q-1$. For $0\le d\le N$, define $P_d$ at the vertex $i-D$ by
\begin{equation}\label{eq:Pd}
 i-D\longmapsto
 \begin{cases}
 i+1,&0\le i<d,\\
 \bot,&i=d,\\
 i,&d<i<q.
 \end{cases}
\end{equation}
All coordinates are modulo $q$. This graph is a spanning partial permutation with one open component from $0$ to $d-D$. If $d=q-s$, translation by $D-d$ identifies it with $R_{D,s}$. Write $j(d)$ for its number of closed components.

The oriented square grid modulo $\Lambda_D$ embeds in the torus $\mathbb R^2/\Lambda_D$; its horizontal and vertical edges represent the steps $D,D+1$. The covolume of $\Lambda_D$ is $q$, since the map $(x,y)\mapsto Dx+(D+1)y$ is onto $\Zq$.

An embedded essential circle on a torus represents a primitive homology class. Indeed, cutting the torus along it yields an annulus, and a transverse arc closes to a curve with intersection number one. Two disjoint essential circles have algebraic intersection number zero, so their primitive classes are parallel. In lattice coordinates the intersection number is $\det(U,V)/q$. A directed cycle in our grid has a nonzero displacement $U=(x,y)\in\Lambda_D$ with $x,y\ge0$, so it is essential. Positivity fixes the sign of the primitive class. Consequently all closed components of a single $P_d$ have the same primitive displacement.

\begin{lemma}[A cut move]\label{lem:cut-move}
For $0\le d<N$, the values $j(d+1),j(d)$ differ by one. If both are positive, the two graphs have the same primitive cycle class.
\end{lemma}
\begin{proof}
Put $v=d-D$ and $w=v+1$. The move from $P_d$ to $P_{d+1}$ deletes $w\to d+1$ and inserts $v\to d+1$. If $w$ is on the open path, this splits off one cycle. If it is on a closed cycle, that cycle joins the open path. All other cycles remain. When both counts are positive, at least one unchanged cycle is present in both graphs, and its class determines the class in each.
\end{proof}

\begin{lemma}[One crossing at a local minimum]\label{lem:local-crossing}
If $1\le d<N$ and $j(d-1)=j(d)+1=j(d+1)$, then $j(d)=0$. The primitive classes $U_-,U_+$ of the two neighboring graphs satisfy
\begin{equation}\label{eq:det}
 \det(U_-,U_+)=q.
\end{equation}
In particular, the slope increases at such a minimum.
\end{lemma}
\begin{proof}
Let $v=d-D$. The cycle removed from $P_{d-1}$ contains the horizontal edge $v\to d$; call it $\Gamma_-$. The cycle added in $P_{d+1}$ contains the vertical edge $v\to d+1$; call it $\Gamma_+$.

Translate the embedded graph $P_{d+1}$ by $\epsilon(e_1-e_2)$ with $0<\epsilon<1$. An original horizontal edge meets a translated vertical edge precisely when their original tails coincide. In~\eqref{eq:Pd}, the first graph's horizontal condition is $i>d-1$, and the second graph's vertical condition is $i<d+1$. Their common index is exactly $i=d$; Figure~\ref{fig:local-crossing} shows the corresponding local lifts.

The other possible crossing is an original vertical edge and a translated horizontal edge with the same original head. Their indices are $i,i+1$ modulo $q$. The required inequalities $i<d-1$ and $i+1>d+1$ are incompatible. The wrap at $q-1$ adds no case, since $i<d-1\le q-3$. Parallel translated edges are disjoint. Thus the graphs have exactly one transverse intersection, at the two specified edges, of positive sign.

It follows that $\Gamma_-$ and the translate of $\Gamma_+$ have intersection number one, giving~\eqref{eq:det}. If $P_d$ contained a closed component, it would remain in both neighbors, forcing its primitive class to be parallel to both $U_-$ and $U_+$. This contradicts~\eqref{eq:det}.
\end{proof}

\begin{figure}[tb]
\centering
\begin{tikzpicture}[x=2.7cm,y=2.4cm,>=Stealth,font=\small]
 \coordinate (V) at (0,0);
 \coordinate (X) at (1,0);
 \coordinate (Y) at (0,1);
 \coordinate (W) at (0.28,-0.28);
 \coordinate (Yp) at (0.28,0.72);
 \draw[gray,dotted] (V)--(Y);
 \draw[thick,->] (V)--(X);
 \draw[thick,densely dashed,->] (W)--(Yp);
 \draw[gray,->] (V)--(W);
 \fill (V) circle (1.2pt);
 \fill (0.28,0) circle (1.4pt);
 \node[below left=3pt] at (V) {$\widetilde v$};
 \node[right=3pt] at (X) {$\widetilde v+e_1$};
 \node[left=3pt] at (Y) {$\widetilde v+e_2$};
 \node[right=5pt] at (W) {$\widetilde v+\epsilon(e_1-e_2)$};
 \node[right=5pt] at (Yp) {$\widetilde v+e_2+\epsilon(e_1-e_2)$};
 \node[above=5pt] at (0.80,0) {$\Gamma_-$};
 \node[right=6pt] at (0.28,0.40) {$\Gamma_+^\epsilon$};
 \node[above right=2pt] at (0.28,0) {$+1$};
\end{tikzpicture}
\caption{The single positive crossing in Lemma~\ref{lem:local-crossing}, shown in a local lift to $\mathbb R^2$. The horizontal edge of $\Gamma_-$ is solid; the translated vertical edge of $\Gamma_+$ is dashed. The dotted edge records its original position. Their oriented intersection contributes $+1$, giving $\det(U_-,U_+)/q=1$.}\label{fig:local-crossing}
\end{figure}
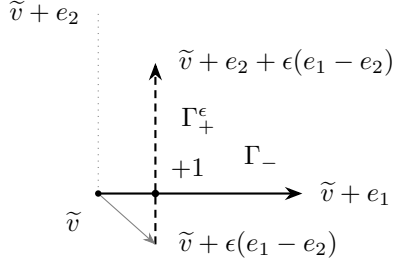

\subsection{The total multiplicity}
Put $H_x=d_0-1$ and $H_y=d_1-1$. For every primitive internal vector $U$ with $\len(U)\le N$, let $H(U)=\lfloor N/\len(U)\rfloor$.

\begin{lemma}\label{lem:residue-count}
The number of points $(x,y)\in\Lambda_D$ with $x,y>0$ and $x+y\le N$ is $(q+1-d_0-d_1)/2$. Equivalently,
\begin{equation}\label{eq:total-multiplicity}
 N=H_x+H_y+2\sum_U H(U).
\end{equation}
\end{lemma}
\begin{proof}
For each length $1\le l<q$, the unique possible second coordinate is
\[
 y_l\equiv-Dl\pmod q,\qquad0\le y_l<q.
\]
An internal point occurs precisely when $0<y_l<l$. There are $d_0-1$ lengths with $y_l=0$ and $d_1-1$ with $y_l=l$, and the two classes are disjoint. On the other lengths, $l\mapsto q-l$ sends $y_l$ to $q-y_l$ and exchanges $y_l<l$ with $y_l>l$. Thus exactly half of the remaining $q+1-d_0-d_1$ lengths contribute an internal point. Grouping the points by their primitive rays gives~\eqref{eq:total-multiplicity}.
\end{proof}

\begin{proof}[Proof of Proposition~\ref{prop:profile}]
At $d=0$, all edges are horizontal except for the missing terminal edge, so $j(0)=H_x$. Similarly $j(N)=H_y$. There is a zero of $j$: either an endpoint is already zero, or the positive endpoint counts have horizontal and vertical classes, and Lemma~\ref{lem:cut-move} prevents a positive sequence of counts from carrying both classes.

On a positive excursion the primitive class $U$ is constant. The open component contains at least one vertex, so $j(d)\len(U)\le N$ and $j(d)\le H(U)$. Lemma~\ref{lem:local-crossing} excludes every positive local minimum. Since the counts change by one, an internal excursion rises to one peak and then descends. At each intervening zero the primitive slope strictly increases. Hence each primitive ray is visited at most once. The initial horizontal excursion decreases from its already maximal height $H_x$ to zero, and the final vertical excursion increases from zero to $H_y$.

Let $h(U)$ be the heights of the internal excursions actually visited. Their total duration gives
\[
 N=H_x+H_y+2\sum_{U\text{ visited}}h(U).
\]
Each $h(U)\le H(U)$, while~\eqref{eq:total-multiplicity} gives the same identity using every internal primitive ray at full height. Thus every ray occurs and every height is attained. Reversing $d=q-s=N-(s-1)$ gives~\eqref{eq:profile}. Its zeros and unit slopes give~\eqref{eq:profile-zeros}; the common primitive class gives the assertion on cycle lengths.
\end{proof}

\begin{samepage}
\section*{Disclosure of generative AI use}
Generative AI systems were used extensively in both the research and the writing of this paper. GPT-6 Pro developed the main proofs, the overall line of argument, and the underlying mathematical ideas, and drafted the manuscript. Claude Opus 5.5 contributed suggestions on the structure and flow of the paper, organization and editing of the text, and auxiliary verification, including independent finite checks of the stated results. The author directed the project, selected the final arguments, checked the mathematical reasoning and the computations, checked the cited literature to the extent used in the paper, and edited the final manuscript. The author assumes full responsibility for the correctness, originality, and integrity of the work. The AI systems are not authors and are not cited as sources of mathematical authority.
\end{samepage}

\enlargethispage{3\baselineskip}

\end{document}